\documentclass[10pt, reqno]{amsart}
\usepackage{hyperref}
\usepackage{amsmath}
\usepackage{amscd}
\usepackage{amsthm}
\usepackage{amssymb}
\usepackage{latexsym}
\usepackage{eufrak}
\usepackage{euscript}
\usepackage{epsfig}
\usepackage{graphics}
\usepackage{array}
\usepackage{enumerate}
\usepackage{dsfont}
\usepackage{color}
\usepackage{wasysym}

\newcommand{\Z}{{\mathbb Z}}
\newcommand{\R}{{\mathbb R}}

\newcommand{\de}{{\delta}}
\newcommand{\ph}{{\varphi}}
\newcommand{\eps}{{\varepsilon}}

\newcommand{\Gm}{{\Gamma}}

\newcommand{\Deg}{{\mathrm{Deg}}}
\newcommand{\supp}{{\mathrm{supp}\,}}

\newtheorem{thm}{Theorem}[section]
\newtheorem{cor}[thm]{Corollary}
\newtheorem{lemma}[thm]{Lemma}
\newtheorem{pro}[thm]{Proposition}

\theoremstyle{definition}
\newtheorem{definition}[thm]{Definition}
\newtheorem{example}[thm]{Example}
\newtheorem{remark}[thm]{Remark}

\newcommand{\Hmm}[1]{\leavevmode{\marginpar{\tiny%
$\hbox to 0mm{\hspace*{-0.5mm}$\leftarrow$\hss}%
\vcenter{\vrule depth 0.1mm height 0.1mm width \the\marginparwidth}%
\hbox to
0mm{\hss$\rightarrow$\hspace*{-0.5mm}}$\\\relax\raggedright #1}}}

\begin{document}

\title{Harmonic functions of general graph Laplacians}

\author{Bobo Hua}
\address{Bobo Hua, Max Planck Institute for Mathematics in the Sciences,
04103 Leipzig, Germany.}
\email{bobohua@mis.mpg.de}

\author{Matthias Keller}
\address{Matthias Keller, Einstein Institute of Mathematics, The Hebrew University of
Jerusalem, 91904 Jerusalem, Israel.}
 \email{mkeller@ma.huji.ac.il}

\begin{abstract}We study harmonic functions on general weighted graphs which allow for a compatible intrinsic metric.
We prove an $L^{p}$ Liouville type theorem which is a quantitative
integral $L^{p}$ estimate of harmonic functions analogous to Karp's
theorem for Riemannian manifolds. As corollaries we obtain Yau's
$L^{p}$-Liouville type theorem on graphs,  identify the domain of
the generator of the semigroup on $L^{p}$ and get a criterion for
recurrence. As a side product, we show an analogue of Yau's $L^{p}$
Caccioppoli inequality. Furthermore, we derive various
 Liouville type results for
harmonic functions on graphs and harmonic maps from graphs into Hadamard spaces.
\end{abstract}
\maketitle
\section{Introduction}
 The study of harmonic functions is a fundamental topic
in {various} areas of mathematics. An important question is which
subspaces of harmonic functions are trivial, that is, they contain
only constant functions. Such results are referred to as Liouville
type theorems. In Riemannian geometry $L^{p}$-Liouville type
theorems for harmonic functions were studied for example  by Yau
\cite{Yau76}, Karp \cite{Karp82b}, Li-Schoen \cite{LiSchoen84} and
many others. Karp's criterion was later generalized by Sturm
\cite{Sturm94} to the setting of strongly local regular Dirichlet
forms. Over the years there were {several} attempts to realize an
analogous theorem for graphs, see Holopainen-Soardi
\cite{HolopainenSoardi97},
Rigoli-Salvatori-Vignati~\cite{RigoliSalvatoriVignati97},
Masamune~\cite{Masamune09} and most recently Hua-Jost
\cite{HuaJost13}. In all these works normalized Laplacians were
studied (often with further restrictions on the vertex degree) and
certain criteria, all weaker than Karp's integral estimate, were
obtained. The main challenge when considering graphs is the
non-existence of a chain rule and, moreover, the fact that for
unbounded graph Laplacians the natural graph distance is very often
not the proper analogue to the Riemannian distance in manifolds. In
this paper, we use the newly developed concept of intrinsic metrics
on graphs to prove an analogue to Karp's theorem for general
Laplacians on weighted graphs. Thus, we generalize all earlier
results on graphs not only with respect to the generality of the
setting but also by recovering the precise analogue of Karp's
criterion.

Harmonic maps are very important nonlinear objects in geometric
analysis studied thoroughly by many authors (e.g. Eells-Sampson
\cite{EellsSampson64}, Schoen-Yau \cite{SchoenYau76},
Hildebrandt-Jost-Widman \cite{HildebrandtJostWidman80}). In this
paper, we adopt a definition of harmonic maps between metric measure
spaces introduced by Jost \cite{Jost94,Jost97,Jost97NPC,Jost98}. In
particular, we study harmonic maps from graphs into Hadamard spaces
(i.e. globally non-positively curved spaces, also called
$\mathrm{CAT}(0)$-spaces), studied also by
\cite{KotaniSunada01,IzekiNayatani05,JostTodjihounde07}, and prove
Liouville type theorems in this context. For various Liouville
theorems on manifolds, we refer to
\cite{Cheng80,Kendall90,Tam95,ChengTamWan96} and references therein.
We prove the finite-energy Liouville theorem  for harmonic maps from
graphs into Hadamard spaces analogous to the one in Cheng-Tam-Wang
\cite{ChengTamWan96} on manifolds.

In what follows we first state and discuss our results and refer for
details and precise definitions to Section~\ref{s:setup}. Our
framework are weighted graphs over a discrete measure
space $(X,m)$ introduced in  \cite{KellerLenz12} which includes non
locally finite graphs, (see also \cite{Soardi08}). In this setting a pseudo metric is called
\emph{intrinsic} if the energy measures of distance functions can be
estimated by the measure of the graph (see
Definition~\ref{d:intrinsic}). We further call such a pseudo metric
\emph{compatible} if it has finite jump size and the weighted vertex degree is bounded on each
distance ball (see
Definition~\ref{d:compatible}). As the boundedness of the weighted
vertex degree is implied by finiteness of distance balls which is
equivalent to metric completeness in the case of a path metric on a
locally finite graph, see
\cite[Theorem~A.1]{HuangKellerMasamuneWojciechowski}, this
assumption can be seen as an analogue of completeness in the
Riemannian manifold case. Similarly, Sturm \cite{Sturm94} asks for
precompactness of balls.

Our first main result is the following analogue to Karp's $L^{p}$ Liouville theorem \cite[Theorem~2.2]{Karp82b}, whose proof  is given
in Section~\ref{s:Karp}. A function is called \emph{(sub)harmonic} if it is
in the  domain of the formal Laplacian and the formal Laplacian
applied to this function is pointwise (less than or) equal to zero,
(see Definition~\ref{d:harmonic}).
We denote by  $1_{B_{r}}$
the characteristic function of the balls $B_{r}$, $r\ge0$, which are taken with respect to an intrinsic metric about a fixed vertex $o\in X$.

\begin{thm}[Karp's $L^{p}$ Liouville theorem] \label{t:Karp}
Assume a connected weighted graph allows for a compatible intrinsic metric. Then every non-negative subharmonic function $f$ satisfying
\begin{align*}
    \inf_{r_{0}>0}\int_{r_{0}}^{\infty} \frac{r}{\|f1_{B_{r}}\|_{p}^{p}}dr=\infty,
\end{align*}
for some $p\in(1,\infty)$, is constant.
\end{thm}

Clearly, the integral in the theorem above diverges, whenever $0\neq f\in L^{p}(X,m)$. Thus, as an immediate corollary, we get Yau's $L^{p}$ Liouville type theorem \cite{Yau76}.

\begin{cor}[Yau's $L^{p}$ Liouville theorem] \label{c:Yau} Assume a connected weighted graph allows for a compatible intrinsic metric. Then every non-negative subharmonic function in $L^{p}(X,m)$, $p\in(1,\infty)$, is constant.
\end{cor}

\begin{remark}\label{r:Karp} (a) The results above imply the corresponding statements for harmonic functions by the simple observation that $f_{+},$ $f_{-}$ and $|f|$ of a harmonic function $f$ are non-negative and subharmonic.

(b) Harmonicity of a function is independent of the
choice of the measure $m$. Hence, for any non-constant harmonic function $f$ on $X,$ we may find a
sufficiently small measure $m$ such that $f\in L^p(X,m)$ for any
$p\in (0,\infty)$, see \cite{Masamune09}. Our theorem states that if we impose the restriction of compatibility on the measure and the metric, then
the $L^p$ Liouville theorem holds for $1<p<\infty.$

(c) Theorem~\ref{t:Karp} generalizes all earlier results on graphs
\cite{HolopainenSoardi97, RigoliSalvatoriVignati97, Masamune09,
HuaJost13} for the case $p\in(1,\infty)$. Not only that our setting
is more general -- as the natural graph distance is always a
compatible intrinsic metric to the normalized Laplacian -- but also
our criterion is more general. In particular, if $f$ satisfies
\begin{align*}
    \limsup_{r\to\infty}\frac{1}{r^2\log r}\|f1_{B_{r}}\|_{p}^{p}<\infty,
\end{align*}
then the integral in Theorem~\ref{t:Karp} diverges. Thus,
Theorem~\ref{t:Karp} is stronger than \cite[Theorem~1.1]{HuaJost13}
(which had only $r^{2}$ rather than $r^{2}\log r$ in the
denominator). The authors of \cite{HuaJost13} observed that for the
normalized Laplacian the case $p\in(1,2]$ can already be obtained by
their techniques, (see  \cite[Remark~3.3]{HuaJost13}). Here, the
missing cases $p\in(2,\infty)$ are treated by {adopting a subtle lemma in}
\cite{HolopainenSoardi97}. Moreover,  our techniques would also
allow a statement such as \cite[Theorem~1.1]{HuaJost13} for the
cases $p<1$.

(d) In \cite{KellerLenz12} discrete measure spaces $(X,m)$ with the
assumption that every infinite path has infinite measure are
discussed (this assumption is denoted by (A) in
\cite{KellerLenz12}). It is not hard to see that for connected
graphs over $(X,m)$ every non-negative subharmonic function
$L^{p}(X,m)$, $p\in[1,\infty)$, is trivial. In fact, from every
non-constant positive subharmonic function we can extract a sequence
of vertices such that the function values increase along this
sequence (compare \cite[Lemma~3.2 and Theorem~8]{KellerLenz12}).
Since this path has infinite measure, the function is not contained
in $L^{p}(X,m)$, $p\in[1,\infty)$. Thus, the only interesting
measure spaces are {those} that contain an
infinite path of finite measure.

(e) Sturm  \cite{Sturm94} proves an analogue for Karp's theorem for
weakly subharmonic functions. This might seem stronger, however, in our setting on graphs weak
solutions of equations are automatically solutions,
\cite[Theorem~2.2 and
Corollary~2.3]{HaeselerKellerLenzWojciechowski12}.
\end{remark}

Corollary~\ref{c:Yau} allows us to explicitly determine the domain
of the generator $L_{p}$ of the semigroup on $L^{p}(X,m)$.  We denote by
$\Delta$ the formal Laplacian with formal domain $F$. (For
definitions see Section~\ref{s:Laplacians}). The proof of the
corollary below is given in Section~\ref{s:corollaries}.

\begin{cor}[Domain of the $L^{p}$ generators] \label{c:domain} Assume a connected weighted graph allows for a compatible intrinsic metric. Then, for $p\in(1,\infty)$, the generator $L_{p}$ is a restriction of $\Delta$ and
\begin{align*}
    D(L_{p})=\{u\in L^{p}(X,m)\cap F\mid \Delta u\in L^{p}(X,m)\}.
\end{align*}
\end{cor}

\begin{remark}(a) The corollary above generalizes \cite[Theorem~1]{HuangKellerMasamuneWojciechowski} to the case $p\in(1,\infty)$ and settles the question in \cite[Remark~3.6]{HuangKellerMasamuneWojciechowski}. Moreover, it complements \cite[Theorem~5]{KellerLenz12}.

(b) It would be interesting to know whether there is a Liouville type theorem for functions in $D(L_{p})$ without the assumption of compatibility on the metric.
\end{remark}

We get furthermore a sufficient criterion for recurrence analogous
to \cite[Theorem~3.5]{Karp82b} and \cite[Theorem~3]{Sturm94} which
generalizes for example \cite[Theorem~2.2]{DoKa88},
\cite[Corollary~B]{RigoliSalvatoriVignati97},
\cite[Lemma~3.12]{Woess00}, \cite[Corollary~1.4]{Grigoryan09}, \cite[Theorem~1.2]{MasamuneUemuraWang12}
on graphs.  For a characterization of recurrence see Proposition~\ref{p:recurrence} in Section~\ref{s:corollaries}, where also the
proof of the corollary below is given.

\begin{cor}[Recurrence]\label{c:recurrence} Assume a connected weighted graph allows for a compatible intrinsic metric. If
\begin{align*}
    \int_{1}^{\infty} \frac{r}{m({B_{r}})}dr=\infty,
\end{align*}
then the graph is recurrent.
\end{cor}

Contrary to the normalized Laplacian, \cite[Theorem~1.2]{HuaJost13},
there is no $L^{1}$ Liouville type theorem in the general case.
However, for stochastic complete graphs (see
Section~\ref{s:Counter-examples}) we have the following analogue to
\cite[Theorem~3]{Grigoryan88}, \cite[Theorem~2]{Sturm94}. The proof
following  \cite{Grigoryan99} is given in
Section~\ref{s:Counter-examples}. We also give counter-examples to
$L^1$ Liouville theorem which complement the counter-examples from
manifolds, \cite{Chung83,LiSchoen84}.

\begin{thm}[Grigor'yan's $L^{1}$ theorem] \label{t:L1Liouville} Assume a connected graph is stochastically complete.  Then, every non-negative superharmonic function in $L^{1}(X,m)$ is constant.
\end{thm}

For vertices $x,y\in X$ that are connected by an edge, we denote a
directed edge by $xy$ and the positive symmetric edge weight by
$\mu_{xy}$. We define
\begin{align*}
    \nabla_{xy}f=f(x)-f(y).
\end{align*}

The following $L^{p}$ Caccioppoli-type inequality is a side product of our analysis. Such an inequality was proven in \cite{HolopainenSoardi97, HuaJost13, RigoliSalvatoriVignati97} for bounded operators. The classical Caccioppoli inequality is the case $p=2$, which can be found for graphs in \cite{CoulhonGrigoryan98,LinXi10,HuangKellerMasamuneWojciechowski}.

\begin{thm}[Caccioppoli-type inequality] \label{t:Caccioppoli} Assume a connected weighted graph allows for a compatible intrinsic metric and $p\in (1,\infty)$. Then, there is $C>0$ such that for  every non-negative subharmonic function $f$ and all $0<r<R-3s$
\begin{align*}
\sum_{x,y\in B_r}\mu_{xy}(f(x)\vee f(y))^{p-2}
|\nabla_{xy}f|^{2}&\leq \frac{C}{(R-r)^{2}}\|f1_{B_{R}\setminus
B_{r}}\|_{p}^{p},
\end{align*}
where $s$ is the jump size of the intrinsic metric (see Section~\ref{s:intrinsic}).
\end{thm}

\begin{remark}\label{r:Caccioppoli} (a) The theorem above allows for a direct proof of Corollary~\ref{c:Yau}, confer  \cite[Corollary~3.1]{HuaJost13}.

(b) For $p\ge2$, we can strengthen the inequality by replacing
$(f(x)\vee f(y))^{p-2}$ on the left hand side by
$f^{p-2}(x)+f^{p-2}(y)$, see Remark~\ref{r:Caccioppoli2} in Section~\ref{s:Caccioppoli}, where the theorem is proven.
\end{remark}

The following  quantitative consequence of Theorem~\ref{t:Karp} which is a generalization of Corollary~\ref{c:Yau}  has various corollaries that are stated and proven in Section~\ref{s:applicationKarp}.
For an intrinsic metric $\rho$ and a fixed vertex $o\in X$ let
$$\rho_{1}=1\vee\rho(\cdot,o).$$

\begin{thm}\label{t:applicationKarp} Assume a connected weighted graph allows for a compatible intrinsic metric $\rho$. If a non-negative subharmonic function $f$ satisfies
$$f\in L^{p}(X,m\rho_{1}^{-2}),$$
for some $p\in(1,\infty)$,
then $f$ is constant.
\end{thm}

Next, we turn to harmonic maps from graphs into Hadamard spaces, (see
Section~\ref{s:harmonicmaps}, in particular Definition~\ref{d:harmonicmaps}). We
prove the following consequence of Karp's theorem in
Section~\ref{s:harmonicmaps}.

\begin{thm}[Karp's theorem for harmonic maps] \label{t:harmonicmaps} Assume a connected weighted graph allows for a compatible intrinsic metric $\rho$.
Let $u$ be a harmonic map into an Hadamard space $(Y,d)$. If there
are $p\in(1,\infty)$ and $y\in Y$ such that
\begin{align*}
   d(u(\cdot),y)\in L^{p}(X,m\rho_{1}^{-2}),
\end{align*}
then $u$ is bounded. Moreover, if $m\rho_{1}^{-2}(X)=\infty$ or $y$ is in the image of $u$,  then $u$ is constant.
\end{thm}

Finally, we turn to harmonic functions and maps of finite energy, (for
definitions see Section~\ref{s:Laplacians} and
Section~\ref{s:harmonicmapsoffiniteenergy}). The two theorems below
stand in close relationship to the celebrated theorem of Kendall
\cite[Theorem~6]{Kendall88}, (confer
\cite{HuangKendall91,KuwaeSturm08}). Our first result in this line
is a direct consequence of Theorem~\ref{t:finiteenergyequivalence}
and it is an analogue to Cheng-Tam-Wang
\cite[Theorem~3.1]{ChengTamWan96}.

\begin{thm}\label{t:boundedfiniteengery} Assume that on a graph every harmonic function of finite energy is bounded. Then, every harmonic map from the graph into an Hadamard space
is bounded.
\end{thm}

The second result in this line is an analogue to
\cite[Theorem~3.2]{ChengTamWan96}. {An Hadamard space is called
locally compact if for any point there exists a precompact
neighborhood.}

\begin{thm}\label{t:boundedfiniteengery2}Assume that on a graph every bounded harmonic function  is constant. Then, every finite-energy harmonic map from the graph into a {locally compact} Hadamard space  is constant.
\end{thm}


The paper is organized as follows. In the next section, we introduce
the involved concepts and recall some basic inequalities.
Section~\ref{s:proofs} is devoted to the proofs of
Theorem~\ref{t:Karp}, Theorem~\ref{t:Caccioppoli} and the
corollaries above.
The proof of Theorem~\ref{t:L1Liouville} and counter-examples to an $L^{1}$-Liouville type statement are given in
Section~\ref{s:Counter-examples}.
In Section~\ref{s:applicationKarp} we prove
Theorem~\ref{t:applicationKarp} and derive various corollaries. Harmonic maps from graphs into
Hadamard spaces are discussed in  Section~\ref{s:harmonicmaps}.
Theorem~\ref{t:harmonicmaps} is proven in
Section~\ref{s:harmonicmapsproof} and
Theorems~\ref{t:boundedfiniteengery} and~\ref{t:boundedfiniteengery2} are proven in
Section~\ref{s:harmonicmapsoffiniteenergy}. Several applications are discussed in Section~\ref{s:HarmonicMapsmeasure}.
\medskip

Throughout this paper $C$ always denotes a constant that  might change from line to line. Moreover, we use the convention that $\infty\cdot0=0$, (which only appears in expressions such as $f^{-q}(x)\nabla_{xy}f$ with $f(x)=f(y)=0$ and $q>0$).


\section{Set-up and preliminaries}\label{s:setup}


\subsection{Weighted graphs}
Let $X$ be a countable discrete set and $m:X\to(0,\infty)$. Extending $m$ additively to sets, $(X,m)$ becomes a measure space with a measure of full support. A
graph over $(X,m)$ is induced by an edge weight function
$\mu:X\times X\to[0,\infty)$, $(x,y)\mapsto\mu_{xy}$ that is
symmetric, has zero diagonal and satisfies
\begin{align*}
    \sum_{y\in X}\mu_{xy}<\infty,\qquad x\in X.
\end{align*}
If $\mu_{xy}>0$ we write $x\sim y$ and let $xy$ and $yx$ be the
oriented edges of the graph.  We write $xy\subset A$ for a set
$A\subseteq X$ if both of the vertices of the edge $xy$ are contained
in $A,$ i.e., $x,y\in A.$  When we fix an orientation for the edges we denote the directed edges often by $e$.

We refer to the triple $(X,\mu,m)$ as a \emph{weighted graph}. We
assume the graph is \emph{connected}, that is for every two
vertices $x,y\in X$ there is a path  $x=x_{0}\sim
x_{1}\sim\ldots\sim x_{n}=y$.

The spaces $L^{p}(X,m)$, $p\in [1,\infty),$ and $L^{\infty}(X)$ are defined in the
natural way. For $p\in[1,\infty)$, let $p^{*}$ be its H\"older dual,
i.e., $\frac{1}{p}+\frac{1}{p^{*}}=1$.


\subsection{Laplacians and (sub)harmonic functions}\label{s:Laplacians}
We define the \emph{formal Laplacian} $\Delta$ on the \emph{formal domain}
\begin{align*}
    F(X)=\{f:X\to\R\mid \sum_{y\in X}\mu_{xy}|f(y)|<\infty\mbox{ for all }x\in X\},
\end{align*}
by
\begin{align*}
    \Delta f(x)=\frac{1}{m(x)}\sum_{y\in X}\mu_{xy}(f(x)-f(y)).
\end{align*}

\begin{definition}[Harmonic function]\label{d:harmonic} A function $f:X\to\R$ is called \emph{harmonic}  (\emph{subharmonic}, \emph{superharmonic}) if $f\in F(X)$ and $\Delta f=0$, ($\Delta f\le0$, $\Delta f\ge0$).
\end{definition}

Obviously, the measure does not play a role in the definition of harmonicity. We denote by $L$ the positive selfadjoint restriction of $\Delta$ on
$L^{2}(X,m)$ which arises from the closure $Q$ of the restriction of the quadratic form $E:\{X\to\R\}\to[0,\infty]$
$$E(f)=\frac{1}{2}\sum_{x,y\in X}\mu_{xy}|\nabla_{xy} f|^{2}$$ to
$C_{c}(X)$, the space of finitely supported
functions, (for details see \cite{KellerLenz12}). Since $Q$ is a
Dirichlet form, the semigroup $e^{-tL}$, $t\ge0$, extends to a
$C_{0}$-semigroup on $L^{p}(X,m)$, $p\in[1,\infty)$ (resp. a weak $C_{0}$-semigroup for $p=\infty$). We denote the generators of these semigroups by $L_{p}$, $p\in[1,\infty)$. Moreover, we say a  function $f$ has \emph{finite energy} if $E(f)<\infty$.


\subsection{Intrinsic metrics}\label{s:intrinsic}

Next, we introduce the concept of intrinsic metrics. A pseudo metric is a symmetric map  $X\times X\to[0,\infty)$ with zero diagonal which satisfies the triangle inequality.

\begin{definition}[Intrinsic metric]\label{d:intrinsic} A pseudo metric $\rho$ on $X$ is called an \emph{intrinsic metric} if
\begin{align*}
\sum_{y\in X}\mu_{xy}\rho^{2}(x,y)\leq m(x),\qquad x\in X.
\end{align*}
\end{definition}

If for a function $f:X\to\R$ the map
$\Gm(f):x\mapsto\sum_{y\in X}\mu_{x y}|\nabla_{xy} f|^{2}$ takes
finite values, then $\Gm(f)$ defines the energy measure of $f$. Thus, a
pseudo metric $\rho$ is intrinsic if the energy measures
$\Gm(\rho(x,\cdot))$, $x\in X$, are absolutely continuous with
respect to $m$ with Radon-Nikodym derivative
$\frac{d}{dm}\Gm(\rho(x,\cdot))=\Gm(\rho(x,\cdot))/m$ satisfying
$\Gm(\rho(x,\cdot))/m\leq1$.

In various situations the natural graph distance proves to be insufficient for the investigations of unbounded Laplacians, see \cite{Wojciechowski1,Wojciechowski2,KellerLenzWojciechowski}. For this reason the concept of intrinsic metrics developed in \cite{FrankLenzWingert12}   for regular Dirichlet forms  received quite some attention as a candidate to overcome these problems. Indeed, intrinsic metrics already have been applied successfully to various problems on graphs \cite{BauerHuaKeller,BauerKellerWojciechowski,FOLZ, FOLZ2,HuangKellerMasamuneWojciechowski} and related settings \cite{GHM}.

The \emph{jumps size} $s$ of a pseudo metric is given by
\begin{align*}
    s:=\sup\{\rho(x,y)\mid x,y\in X, x\sim y\}\in[0,\infty].
\end{align*}

From now on, $\rho$ always denotes an intrinsic metric and $s$ denotes its jump size.\medskip

We fix a base point  $o\in X$ which we suppress in notation and
denote the distance balls by
\begin{align*}
    B_{r}=\{x\in X\mid\rho(x,o)\leq r\},\qquad r\ge0.
\end{align*}
Since $\rho$ takes values in $[0,\infty)$ in our setting, the results are indeed
independent of the choice of $o$. For $U\subseteq X$, we write
$B_{r}(U)=\{x\in X\mid\rho(x,y)\leq r \mbox{ for some }y\in U\}$,
$r\geq 0$.

Define the \emph{weighted vertex degree} $\mathrm{Deg}:X\to[0,\infty)$ by
\begin{align*}
    \mathrm{Deg}(x)=\frac{1}{m(x)}\sum_{y\in X}\mu_{xy},\qquad x\in X.
\end{align*}

\begin{definition}[Compatible metric]\label{d:compatible} A pseudo metric on $X$ is called \emph{compatible} if it has finite jump size and the restriction of $\mathrm{Deg}$ to every distance ball is bounded, i.e., $\mathrm{Deg}|_{B_r}\leq C(r)<\infty$ for all $r\ge0$.
\end{definition}

\begin{example}\label{ex:intrinsic}
(a) For any  given weighted graph there is an intrinsic path metric defined by
\begin{align*}
    \de(x,y)=\inf_{x=x_{0}\sim\ldots\sim x_{n}=y}\sum_{i=0}^{n-1} (\mathrm{Deg}(x_{i})\vee\mathrm{Deg}(x_{i+1}))^{-\frac{1}{2}}.
\end{align*}
This intrinsic metric can be turned into an intrinsic metric
$\de_{r}$ with finite jump size $s=r$ by taking the path metric with
edge weights $\de(x,y)\wedge r$, $x\sim y$. In many cases, neither
$\de_{r}$ nor $\de$ is compatible.

(b) If the measure $m$ is larger than  the measure $n(x)=\sum_{y\in
X}\mu_{xy}$, $x\in X$, then the natural graph distance (i.e., the
path metric with edge weights $1$) is an intrinsic metric which is
compatible since $s=1$ and $\mathrm{Deg}\le1$ in this case.
\end{example}

\begin{remark}(a)
In view of Example~\ref{ex:intrinsic}~(b) it is apparent that  \cite[Theorem~1.1]{HuaJost13} is included in Theorem~\ref{t:Karp}.

(b) In \cite[Theorem~A.1]{HuangKellerMasamuneWojciechowski} a Hopf-Rinow type
theorem is shown which states that for a locally finite graph a path
metric is complete if and only if all balls are finite. Thus, compatibility can be seen as
a completeness assumption of the graph.

(c) It is not hard to see that there are graphs that do not allow
for a compatible intrinsic metric. However, to a given edge weight
function $\mu$ and a pseudo metric $\rho$, we can always assign a
minimal measure $m$ such that $\rho$ is intrinsic, i.e., let
$m(x)=\sum_{y\in X}\mu_{xy}\rho^{2}(x,y)$, $x\in X$. If $\rho$
already has finite jump size and all balls are finite, then $\rho$
is automatically compatible.

(d) The assumption that $\mathrm{Deg}$ is bounded on distance balls is equivalent to either of the following assumptions
\begin{itemize}
  \item [(i)] The restriction of $\Delta$ to any distance ball (with Dirichlet boundary conditions) is a bounded operator.
  \item [(ii)] The Radon-Nikodym derivative of the  measure $n$ given by   $n(x)=\sum_{y}\mu_{xy}$, $x\in X$, with respect to the measure $m$ is bounded on the distance balls .
\end{itemize}
The equivalence of (i) follows from Theorem~\cite[Theorem~9.3]{HaeselerKellerLenzWojciechowski12}
and the one of (ii) is obvious.
\end{remark}

In the subsequent, we will make use of the  cut-off function $\eta=\eta_{r,R}$, $0\leq r<R$, on $X$ given by
\begin{align*}
    \eta=1\wedge \Big(\frac{R-\rho(\cdot,o)}{R-r}\Big)_{+}.
\end{align*}


\begin{lemma}\label{l:cutoff}Let $\eta=\eta_{r,R}$, $0<r<R$, be given as above. Then,
\begin{itemize}
  \item [(a)] $\eta|_{B_{r}}\equiv 1$ and $\eta|_{X\setminus B_{R}}\equiv0$.
  \item [(b)] For $x\in X$,
  $$\sum_{y\in X}\mu_{xy}|\nabla_{xy}\eta|^{2}\leq \frac{1}{(R-r)^{2}}1_{B_{R+s}\setminus B_{r-s}}(x)m(x).$$
\end{itemize}
\end{lemma}
\begin{proof}(a) is obvious from the definition of $\eta$ and (b) follows directly from $|\nabla_{xy}\eta|\leq \frac{1}{R-r}\rho(x,y)1_{B_{R+s}\setminus B_{r-s}}(x)$ for $x\sim y$ and the intrinsic metric property of $\rho$.
\end{proof}


\subsection{Green's formula, Leibniz rules and mean value {theorem}}
We first prove a Green's formula which is an $L^{p}$ version of the one in \cite{HuangKellerMasamuneWojciechowski}.

\begin{lemma}[Green's formula]\label{l:Green}Let $p\in[1,\infty)$, $U\subseteq X$ and assume $\mathrm{Deg}$ is bounded on $U$. Then for all $f$ with $f1_{U}\in L^{p}(X,m)\cap F(X)$ and  $g\in L^{^{p^{*}}}(X,m)$ with $B_{s}(\supp g)\subseteq U$
\begin{align*}
   \sum_{x\in X}(\Delta f)(x)g(x)m(x)=\frac{1}{2}\sum_{x,y\in U}\mu_{xy} \nabla_{xy}f \nabla_{xy}g.
\end{align*}
\end{lemma}
\begin{proof} The formal calculation in the proof of Green's formula is a straightforward algebraic manipulation. To ensure that all involved terms converge absolutely, one invokes H\"older's inequality and the boundedness assumption on $\Deg$ (confer the proof of Lemma~3.1 and~3.3 in \cite{HuangKellerMasamuneWojciechowski}).
\end{proof}

The following Leibniz rules follow by direct computations.
\begin{lemma}[Leibniz rules]\label{l:Leibniz} For all $x,y\in X$, $x\sim y$ and $f,g:X\to\R$
\begin{align*}
    \nabla_{xy}(fg)
    &=f(y)\nabla_{xy}g+g(x)\nabla_{xy}f\\
    &= f(y)\nabla_{xy}g+g(y)\nabla_{xy}f +\nabla_{xy}f\nabla_{xy}g.
\end{align*}
\end{lemma}

A fundamental difference of Laplacians on graphs and on manifolds is
the absence of a chain rule in the graph case. In particular,
existence of a chain rule can be used as a characterization for a
regular Dirichlet form to be strongly local. We circumvent this
problem by using the mean value {theorem} from
calculus instead. In particular, for a continuously differentiable
function $\phi:\R\to\R$ and $f:X\to\R$, we have
\begin{align*}
    \nabla_{xy} (\phi\circ f)=\phi'(\zeta)\nabla_{xy}f,\qquad\mbox{for some }\zeta\in[f(x)\wedge f(y),f(x)\vee f(y)].
\end{align*}
In this paper we will apply this formula to get estimates  for the function $\phi:t\mapsto t^{p-1}$, $p\in(1,\infty)$. However, we need a refined inequality as it was already used in the proof of \cite[Theorem~2.1]{HolopainenSoardi97}. For the convenience of the reader, we include a short proof here.

\begin{lemma}[Mean value {inequalities}]\label{l:MVI} For all $f:X\to\R$ and $x\sim y$ with $\nabla_{xy}f\ge0,$
\begin{itemize}
  \item [(a)] $\nabla_{xy}f^{p-1}\ge \frac{1}{2}(f^{p-2}(x)+f^{p-2}(y))\nabla_{xy}f$, for $ p\in[2,\infty),$
  \item [(b)] $\nabla_{xy}f^{p-1}\ge C (f(x)\vee f(y))^{p-2}\nabla_{xy}f$, for $ p\in(1,\infty)$, where $C=(p-1)\wedge 1$.
\end{itemize}
\end{lemma}
\begin{proof}(a) Denote $a=f(y)$, $b=f(x)$. As it is the only non-trivial case, we assume $0<a<b$. Note that for $p\neq 1$
\begin{equation*}
b^{p-1}-a^{p-1}=(b-a)(b^{p-2}+a^{p-2})+ab(b^{p-3}-a^{p-3}).
\end{equation*}
Thus, the statement is immediate for $p\ge3$ since the second term
on the right side is non-negative in this case. Let $2\leq p<3$ and
note $a^{p-3}>b^{p-3}$. The function $t\mapsto t^{2-p}$ is convex on
$(0,\infty)$ and, thus, its image lies below the line segment
connecting $(b^{-1},b^{p-2})$ and $(a^{-1},a^{p-2})$. Therefore,
\begin{align*}
{a^{p-3}-b^{p-3}}&\leq   \frac{a^{p-3}-b^{p-3}}{(3-p)}=\int_{b^{-1}}^{a^{-1}}t^{2-p}dt\leq (a^{-1}-b^{-1})\Big(\frac{(b^{p-2}-a^{p-2})}{2}+a^{p-2}\Big)\\
&=\frac{1}{2ab}(b-a)(a^{p-2}+b^{p-2}).
\end{align*}
From the  equality in the beginning of the proof we now deduce the assertion in the case $2\leq p<3$.\\
(b) The case $p\ge2$ follows from (a). The case $1<p\leq 2$ in (b) follows directly from the mean value theorem.
\end{proof}


\section{Proofs for harmonic functions}\label{s:proofs}
In this section we prove the main theorems and the corresponding
corollaries for harmonic functions. It will be convenient to
introduce the following orientation on the edges. For a given
non-negative subharmonic function $f$, we let $E_{f}$ be the
set of
oriented edges $e=e_{+}e_{-}$ such that
\begin{align*}
    \nabla_{e}f\ge0,\quad\mbox{i.e., }\,f(e_{+})\ge f(e_{-}).
\end{align*}


\subsection{The key estimate}
The lemma below is vital for the proof of Theorem~\ref{t:Karp} and Theorem~\ref{t:Caccioppoli}.

\begin{lemma}\label{l:keyestimate}
Let $p\in (1,\infty)$, $0\leq\ph\in L^{\infty}(X)$ and $U=B_{s}(\supp\ph)$. Assume $\mathrm{Deg}$ is bounded on $U$. Then, for every non-negative subharmonic function $f$ with $f1_{U}\in L^{p}(X,m)$,
\begin{align*}
\sum_{e\in E_{f}}\mu_{e} f^{p-2}(e_{+})\ph^{2}(e_{-})|\nabla_{e}f|^{2} &\leq C \sum_{e\in E_{f},e\subset U} \mu_{e}f^{p-1}(e_{+})\ph(e_{-})\nabla_{e}f|\nabla_{e}\ph|,
\end{align*}
where $C=2/((p-1)\wedge 1)$.
\end{lemma}
\begin{proof}
From the assumptions  $f1_{U}\in L^{p}(X,m)$ and $\ph\in
L^{\infty}(X)$, we infer $\ph^2 f^{p-1}\in {L^{p^{*}}(X,m)}$ (as
$p^{*}=p/(p-1)$). Thus, compatibility of the pseudo metric implies
applicability of Green's formula with $f$ and $g=\ph^2 f^{p-1}$. We
start by using non-negativity and  subharmonicity of $f$ before
applying Green's formula (Lemma~\ref{l:Green}) and  the first and
second Leibniz rule (Lemma~\ref{l:Leibniz})
\begin{align*}
    0&\geq \sum_{x\in X}(\Delta f)(x)(\ph^{2} f^{p-1})(x)m(x)=\sum_{e\in E_{f},e\subset U}\mu_{e}\nabla_e f\nabla_{e}(\ph^{2} f^{p-1})\\
    &=\sum_{e\subset U}\mu_{e}\nabla_{e} f\big[\ph^{2}(e_{-})\nabla_{e} f^{p-1}+f^{p-1 }(e_{+})\nabla_{e}\ph^{2}\big]\\
    &=\sum_{e\subset U}\mu_{e}\nabla_{e} f\big[\ph^{2}(e_{-})\nabla_{e} f^{p-1}+2f^{p-1}(e_{+})\ph(e_{-})\nabla_{e}\ph +f^{p-1}(e_{+})|\nabla_{e}\ph|^{2}\big]\\
    &\geq C\sum_{e\subset U} \mu_{e}f^{p-2}(e_{+})\ph^{2}(e_{-}) |\nabla_{e}f|^{2}+2\sum_{e\subset U} \mu_{e}f^{p-1}(e_{+})\ph(e_{-})\nabla_{e} f\nabla_{e}\ph,
\end{align*}
where we dropped the third term in the third line since it is
non-negative because of  $\nabla_{e}f\ge0$  and we estimated the
first term on the right hand side using the mean value
{theorem}, Lemma~\ref{l:MVI}~(b). Absolute convergence of the
two terms in the last line can be checked using H\"older's
inequality and the assumptions $f1_{U}\in L^{p}(X,m)$, $\ph\in
L^{\infty}(X)$ and boundedness of $\mathrm{Deg}$ on $U$. Hence, we
obtain the statement of the lemma.
\end{proof}

\subsection{Proof of Karp's theorem}\label{s:Karp}

\begin{proof}[Proof of Theorem~\ref{t:Karp}]
Let $p\in(1,\infty)$ and let $f$ be a non-negative subharmonic function. Assume $f1_{B_{r}}\in L^{p}(X,m)$ for all $r\ge0$ since otherwise $\inf_{r_{0}}\int_{r_{0}}^{\infty} r/\|f1_{B_{r}}\|_{p}^{p}dr=0$.
Let $\eta=\eta_{r+s,R-s}$ with $0<r<R-3s$ (see
Section~\ref{s:intrinsic}). Then by  Lemma~\ref{l:keyestimate}
(applied with $\ph=\eta$) we obtain (noting additionally that $\nabla_{xy}\eta=0$, $x,y\in B_{r}$)
\begin{align*}
    \sum_{e\subset B_{R}} & \mu_{e}f^{p-2}(e_{+})\eta^{2}(e_{-}) |\nabla_{e}f|^{2} \leq C\sum_{e\subset B_{R}\setminus B_{r}} \mu_{e}f^{p-1}(e_{+})\eta(e_{-})\nabla_{e}f|\nabla_{e}\eta|.
\end{align*}
Now, the Cauchy-Schwarz inequality, $\sum_{e}\mu_{e}f^{p}(e_{+})|\nabla_{e}\eta|^{2}\leq \sum_{x,y}\mu_{xy}f^{p}(x)|\nabla_{xy}\eta|^{2}$ and the cut-off function lemma, Lemma~\ref{l:cutoff}, yield
\begin{align*}
\Big(  \sum_{e\subset B_{R}} & \mu_{e}f^{p-2}(e_{+})\eta^{2}(e_{-}) |\nabla_{e}f|^{2}  \Big)^{2} \\
&\leq C\Big(\sum_{e\subset B_{R}\setminus B_{r}}\mu_{e} f^{p}(e_{+})|\nabla_{e}\eta|^{2}\Big)\Big(\sum_{e\subset B_{R}\setminus B_{r}} f^{p-2}(e_{+})\eta^{2}(e_{-}) |\nabla_{e}f|^{2}\Big) \\
&\leq \frac{C}{(R-r)^{2}}\|f1_{B_{R}\setminus B_{r}}\|_{p}^{p}\left(\Big(\sum_{e\subset B_{R}}- \sum_{e\subset  B_{r}} \Big) f^{p-2}(e_{+})\eta^{2}(e_{-})|\nabla_{e}f|^{2}\right).
\end{align*}
Let $R_0\geq 3s$ be such that $f 1_{B_{R_{0}}}\neq 0$ and  denote $$v(r)=\|f1_{B_{r}}\|_{p}^{p},\qquad r\ge0.$$ Moreover, for $j\ge0$, let $R_j=2^{j}R_{0}$, $\ph_{j}=\eta_{R_{j}+s,R_{j+1}-s}$ and
\begin{align*}
Q_{j+1}&=\sum_{e\subset B_{R_{j+1}}}
\mu_{e}f^{p-2}(e_{+})\ph_{j}^{2}(e_{-}) |\nabla_{e}f|^{2}.
\end{align*}
As $\ph_{j-1}\leq\ph_{j}$, we get  $Q_{j}\leq
Q_{j+1}$ and together with the estimate above this implies
\begin{align*}
    Q_jQ_{j+1}\leq Q_{j+1}^2\leq C\frac{v(R_{j+1})}{(R_{j+1}-R_{j})^2}(Q_{j+1}-Q_{j}),\qquad j\ge0.
\end{align*}
Since $R_{j+1}=2R_{j}$, dividing the above inequality by
$\frac{v(R_{j+1})}{R_{j+1}^2}Q_{j}Q_{j+1}$ and adding $C/Q_{j+1}$
yield
\begin{align*}
\frac{R_{j+1}^{2}}{v(R_{j+1})}+\frac{C}{Q_{j+1}}\leq \frac{C}{Q_{j}}
\end{align*}
and, thus,
\begin{align*}
 \frac{1}{C}\sum_{j=1}^{\infty}\frac{R_{j+1}^{2}}{v(R_{j+1})}\leq \frac{1}{Q_{1}}.
\end{align*}
Now, the assumption $\int_{R_{0}}^{\infty}r/v(r)dr=\infty$ implies  $\sum_{j=0}^{\infty}\frac{R_{j}^{2}}{v(R_{j})}=\infty$. Therefore, $Q_{1}=0$. As this is true for all  $R_{0}$ large enough, we have
\begin{align*}
    f^{p-2}(e_{+}) |\nabla_{e}f|^{2}=0,
\end{align*}
for all edges $e$. For $p\ge2$, connectedness clearly implies that $f$ is constant. On the other hand, for $p\in(1,2]$, we always have $f^{p-2}(e_{+})>0$ and, thus,  $f$ is constant.
\end{proof}
\subsection{Proof of the Caccioppoli inequality}\label{s:Caccioppoli}

\begin{proof}[Proof of Theorem~\ref{t:Caccioppoli}]
Using Lemma~\ref{l:keyestimate} and the inequality $ab\leq \eps
a^{2}+\frac{1}{4\eps} b^{2}$, $\eps>0$, we estimate
\begin{align*}
\sum_{e\in E_{f}} \mu_{e}f^{p-2}(e_{+})\ph^{2}(e_{-})&|\nabla_{e}f|^{2}\leq C \sum_{e\in E_{f}} \mu_{e}f^{p-1}(e_{+})\ph(e_{-})\nabla_{e}f|\nabla_{e}\ph|\\
&\leq \frac{1}{2} \sum_{e\in E_{f}} \mu_{e}f^{p-2}(e_{+})\ph^{2}(e_{-}) |\nabla_{e}f|^{2}+C\sum_{e\in E_{f}} \mu_{e} f^{p}(e_{+})|\nabla_{e}\ph|^{2}.
\end{align*}
Letting $\ph=\eta=\eta_{r+s,R-s}$ with $0<r<R-3s$ (from Section~\ref{s:intrinsic}) and using the cut-off function lemma, Lemma~\ref{l:cutoff}, we arrive at
\begin{align*}
  \sum_{e\in E_{f}} \mu_{e}f^{p-2}(e_{+})|\nabla_{e}f|^{2}  &\leq C\sum_{e\in E_{f}} \mu_{e} f^{p}(e_{+})|\nabla_{e}\eta|^{2}
  \leq \frac{C}{(R-r)^{2}}\|f1_{B_{R}\setminus B_{r}}\|_{p}^{p}.
\end{align*}
\end{proof}

\begin{remark}\label{r:Caccioppoli2} In order to obtain the stronger statement for $p\in[2,\infty)$ mentioned in Remark~\ref{r:Caccioppoli}~(b), we  invoke Lemma~\ref{l:MVI}~(a) in the proof of Lemma~\ref{l:keyestimate} instead of Lemma~\ref{l:MVI}~(b) and proceed as in the proof above.
\end{remark}
\subsection{Proof of the corollaries}\label{s:corollaries}

In this section we prove the corollaries.

\begin{proof}[Proof of Corollary~\ref{c:Yau} (Yau's $L^{p}$ Liouville theorem)] Clearly the integral in Theorem~\ref{t:Karp} diverges if $f\in L^{p}(X,m)$.
\end{proof}

\begin{proof}[Proof of Corollary~\ref{c:domain} (Domain of the $L^{p}$ generators)]
Let $f\in L^{p}(X,m)\cap F(X)$ be such that $(\Delta+1)f=0$. Since
the positive and negative part $f_{+}$, $f_{-}$ of $f$ are
non-negative, subharmonic and in $L^{p}(X,m)$, they must be constant
by Corollary~\ref{c:Yau}. This implies $f_{\pm}\equiv0$ and, thus,
$f\equiv 0$. Now, the proof of the corollary works {literally
line by line} as  the proof of
\cite[Theorem~5]{KellerLenz12}.
\end{proof}

For the proof of Corollary~\ref{c:recurrence}  we recall the
following well known equivalent conditions for recurrence.

\begin{pro}[Characterization of recurrence]\label{p:recurrence} Let a connected graph $X$ be given. Then the following are equivalent.
\begin{itemize}
  \item [(i)] For the transition matrix $P$ with $P_{x,y}=\mu_{xy}/\sum_{z\in X}\mu_{xz}$, $x,y\in X$, we have
      $        \sum_{n=0}^{\infty}P^{(n)}(x,y)=\infty$
      for some (all) $x,y\in X$, where $P^{(n)}$ denotes the $n$-th power of $P$.
  \item [(ii)]  For $m\equiv 1$ and some (all) $x,y\in X$, we have $    \int_{0}^{\infty}e^{-tL}\de_{x}(y) dt=\infty$, where $\de_{x}(y)=1$ if $x=y$ and zero otherwise.
  \item [(iii)]  For all $m$ and some (all) $x,y\in X$,
we have
$    \int_{0}^{\infty}e^{-tL}\de_{x}(y) dt=\infty$.
  \item [(iv)] Every  bounded superharmonic (or subharmonic) function is constant.
  \item [(v)]  Every  non-negative superharmonic function is constant.
  \item [(vi)] Every  superharmonic (or subharmonic) function of finite energy is constant.
  \item [(vii)] $\mathrm{cap}(x):=\inf\{E(f)| f\in C_c(X), f(x)=1\}=0$ for
  some (all) $x\in X$
\end{itemize}
\end{pro}
A graph is called \emph{recurrent} if one of the equivalent statements of Proposition~\ref{p:recurrence} is satisfied.

\begin{proof} The equivalence (i)$\Leftrightarrow$(ii) is shown in \cite[Theorem~6]{Schmidt12} (confer \cite[Theorem~4.34]{Chen04}).
The equivalences (ii)$\Leftrightarrow$(vi)$\Leftrightarrow$(iii)
are in \cite[Theorem~2 and Theorem~9]{Schmidt12} (confer
\cite[Theorem~3.34]{Soardi08}). The equivalences
(i)$\Leftrightarrow$(v)$\Leftrightarrow$(vii) are found in
\cite[Theorem~1.16, Theorem~2.12]{Woess00}. The equivalence
(iv)$\Leftrightarrow$(v) follows since every non-negative
superharmonic function $f$ can be approximated by the bounded
superharmonic functions $f\wedge n$, $n\ge1$.
\end{proof}

\begin{proof}[Proof of Corollary~\ref{c:recurrence} (Recurrence)]
Theorem~\ref{t:Karp} implies that any non-negative bounded
subharmonic function $f$ is constant provided
$\inf_{r_{0}}\int_{r_{0}}^{\infty}r/m(B_{r})dr=\infty$ since $\|f
1_{B_{r}}\|_{p}^{p}\leq \|f\|^p_{\infty}m(B_r)$, $r\ge0$. By
Proposition~\ref{p:recurrence} the graph is recurrent.
\end{proof}


\section{$L^{1}$-Liouville theorem and counter-examples}\label{s:Counter-examples}

In this section we deal with the borderline case of the $L^p$ Liouville theorem $p=1.$
We first prove Theorem~\ref{t:L1Liouville} which deals with the stochastic complete case and then  give two examples which show that
there is no $L^1$ Liouville theorem for non-negative subharmonic functions in the general case.

A graph is called \emph{stochastically complete} if $e^{-tL}1=1$, where $1$ denotes the function that is constantly one on $X$. For the relevance of the concept see \cite{Grigoryan99,KellerLenz12,Wojciechowski1}.
The proof of Theorem~\ref{t:L1Liouville} follows along the lines of the proof of \cite[Theorem~13.2]{Grigoryan99}.

\begin{proof}[Proof of Theorem~\ref{t:L1Liouville}]If the graph is recurrent, then there are no non-constant non-negative superharmonic functions by
Proposition~\ref{p:recurrence}. So assume the graph is not recurrent
which implies $G(x,y)=\int_{0}^{\infty}e^{-tL}\de_{x}(y)dt<\infty$,
$x,y\in X$, again by Proposition~\ref{p:recurrence}. Let $K_{n}$, $n\ge0$, be an  sequence of finite sets  exhausting $X$  and
$G_{n}(x,y)=\int_{0}^{\infty}e^{-tL_{n}}\de_{x}(y)dt$, where $L_{n}$
are the finite dimensional operators arising from the restriction of
the form $Q$ to $C_{c}(K_{n})$. By domain monotonicity,
\cite[Proposition~2.6 and~2.7]{KellerLenz12} the semigroups
$e^{-tL_{n}}$ converge monotonously increasing to $e^{-tL}$ and,
hence, $G_{n}(x,y)\leq G(x,y)$ for $x,y\in K_{n}$, and $G_{n}\nearrow
G$, $n\to\infty$, pointwise. By direct calculation for any $x\in K_n$
\begin{align*}
    L_{n}G_{n}(x,y) =\int_{0}^{\infty}L_{n}e^{-tL_{n}}\de_{x}(y)dt=
    \int_{0}^{\infty}\partial_{t}e^{-tL_{n}}\de_{x}(y)dt =[e^{-tL_{n}}\de_{x}(y)]_{0}^{\infty}=\de_{x}(y)
\end{align*}
and, hence,  $G_{n}(x,\cdot)$ are harmonic on $K_{n}\setminus \{x\}$, $n\ge0$.

Let $u$ be a non-trivial non-negative superharmonic function which is strictly positive  by the minimum principle \cite[Theorem~8]{KellerLenz12}.
Let $U\subseteq X$ be finite with $o\in U\subseteq K_{n}$, $n\ge0$ and $C>0$ be such that $Cu\ge G(o,\cdot)$ on $U$. By the minimum principle $Cu\ge G_{n}(o,\cdot)$ on $K_{n}\setminus \{o\}$ and, hence,
$Cu\ge G(o,\cdot)$ on $X$ by the discussion above. If the graph is stochastically complete, then we get by Fubini's theorem
\begin{align*}
C  \|u\|_{1} \ge\|G(o,\cdot)\|_{1}=\int_{0}^{\infty}\sum_{x\in X}e^{-tL}\de_{o}(x)m(x)dt=\int_{0}^{\infty}e^{-tL}1(o)dt =\int_{0}^{\infty}dt=\infty.
\end{align*}
Hence, $u$ is not in $L^{1}(X,m)$.
\end{proof}
In the proof we show that in the non-recurrent case there are no
nontrivial superharmonic functions in  $L^{1}$. This is explained
since in the case of finite measure recurrence and stochastic
completeness are equivalent \cite[Theorem~12]{Schmidt12}.

Next, we show that in general there is no $L^{p}$ Liouville theorem
for $p\in(0,1].$ This is analogous to the situation in  Riemannian
geometry, where counter-examples were given by
\cite{Chung83,LiSchoen84}. Our first example is a graph of finite
volume and the second is of infinite volume.

\begin{example}[Finite volume]
Let $G=(X,\mu,m)$ be an infinite line graph, i.e., $X=\Z$ and $x\sim y$ iff
$|x-y|=1$ for $x,y\in \Z$. Define the edge weight by
$\mu_{xy}=2^{1-(|x|\vee |y|)}$ for $x\sim y$ and the measure $m$ by
$m(x)=(|x|+1)^{-2}2^{-|x|}$, $x\in\Z$, which implies
$m(X)<\infty$. The intrinsic metric $\de$ (introduced in
Example~\ref{ex:intrinsic}) is compatible as it satisfies
$\de(x,x+1)\ge C(|x|+1)^{-1}$ and, thus, $\sum_{x=-\infty}^{\infty}\de(x,x+1)=\infty$.
However, the function $f$ defined as
\begin{align*}
    f(x)=\mathrm{sign}(x)(2^{|x|}-1),\quad x\in\Z,
\end{align*}
is harmonic and,  clearly, $f\in L^{p}(X,m)$, $p\in(0,1]$.
\end{example}

\begin{example}[Infinite volume]
We can extend the example above to the infinite volume case. Let $G$
be the graph from above and $G'$ be a locally finite graph of infinite volume which allows for a compatible path metric. We glue
$G'$  to the vertex $x=0$ of the graph $G$ by identifying a vertex in $G'$ with $x=0$. Next, we extend the path metrics in the  natural way and obtain (by renormalizing the edge weights of the metric at the edges around $x=0$ if necessary) again a compatible intrinsic metric and
the graph has infinite volume. Moreover, we extend $f$ on $G$ from above by zero to $G'$ and obtain a harmonic function which is in $L^{p}$, $p\in(0,1]$.
\end{example}


\section{Applications of Karp's theorem}\label{s:applicationKarp}

In this section we prove Theorem~\ref{t:applicationKarp} and give several applications which mainly circle around the case of finite measure.

\begin{proof}[Proof of Theorem~\ref{t:applicationKarp}]
We assume that for some $p\in (1,\infty)$ the non-negative subharmonic function $f$ is in $L^{p}(X,m \rho_{1}^{-2})$ and, hence, $f^{p}\rho_{1}^{-2}\in L^{1}(X,m)$.
For large $r_{0}\ge 1$, we estimate
\begin{align*}
    \int_{r_{0}}^{\infty} \frac{r}{\|f1_{B_{r}}\|_{p}^{p}}dr
    \geq
    \int_{r_{0}}^{\infty} \frac{r}{r^2\|f^{p}\rho^{-2}_{1}1_{B_{r}}\|_{1}} dr \ge C\int_{r_{0}}^{\infty} \frac{1}{r}dr=\infty.
\end{align*}
Hence, Theorem~\ref{t:Karp} implies that $f$ is constant.
\end{proof}

Next, we turn to several consequences of
Theorem~\ref{t:applicationKarp}. A function $f:X\to\R$ is said to
\emph{grow less than} a function $g:[0,\infty)\to(0,\infty)$ if
there are $\beta\in(0,1)$ and $C>0$ such that
$$f(x)\leq C g^{\beta}(\rho_{1}(x)),\quad x\in X. $$
We say $f$ \emph{grows  polynomially} if $f$ grows less than a polynomial.

We say the measure $m$ has a \emph{finite $q$-th moment}, $q\in \R$, with respect to an intrinsic metric $\rho$ if
$$\rho_{1}\in L^{q}(X,m),$$
where $\rho_{1}=1\vee\rho(\cdot,o)$.
This assumption implies that all balls have finite measure and if $q\ge0$ it also implies $m(X)<\infty$.

\begin{cor}[Measures with finite moments] \label{c:finitemomentmeasure} Assume a connected weighted graph allows for a compatible intrinsic metric and the measure has a finite $q$-th moment, $q\in\R$.
 Then every non-negative subharmonic function $f$ that grows less than $r\mapsto  r^{q+2}$
is constant. In particular, if $q>-2$, then boundedness of $f$ implies $f$ is constant.
\end{cor}
\begin{proof} If $f$ grows less than $r\mapsto r^{q+2}$, then there is $\eps>0$ such that
$f^{1+\eps}\rho_{1}^{-2}\leq C \rho_{1}^{q}$ on $X$. By the assumption
$\rho_{1}\in L^{q}(X,m)$ it follows $f\in L^{p}(X,m\rho_{1}^{-2})$
for $p=1+\eps$. Hence, the assertion follows from
Theorem~\ref{t:applicationKarp}.
\end{proof}

Letting $q=0$ in the above theorem gives the following immediate corollary.

\begin{cor}[Finite measure] \label{c:finitemeasure} Assume a connected weighted graph allows for a compatible intrinsic metric and $m(X)<\infty.$ Then every non-negative subharmonic function $f$ that grows less than quadratic
is constant. In particular, $f\in L^{\infty}(X)$ implies that $f$ is constant.
\end{cor}

The final corollary of this section is a consequence of Corollary~\ref{c:Yau}.
\begin{cor}[Exponentially decaying measure]\label{c:decayingmeasure} Assume a connected weighted graph allows for a compatible intrinsic metric and {$m(X)<\infty,$} and there is $\beta>0$ such
that
\begin{align*}
    \limsup_{r\to\infty}\frac{1}{r^{\beta}}\log m(B_{r+1}\setminus B_r) <0.
\end{align*}
Then  every non-negative subharmonic function that grows polynomially
is constant.
\end{cor}
\begin{proof}
If a non-negative subharmonic function  $f$ grows polynomially, then there is $q>0$ such that
{\begin{align*}
    \|f\|_{p}^{p}\leq C \sum_{x\in X}\rho_1^{q}(x,o)m(x)\leq C \sum_{r=1}^{\infty}r^{q}m(B_{r}\setminus B_{r-1})+C<\infty
\end{align*}} by the assumption on the measure. Hence, the theorem follows
from Corollary~\ref{c:Yau}.
\end{proof}


\section{Applications to harmonic maps}\label{s:harmonicmaps}

Harmonic maps between metric measure spaces were introduced by Jost \cite{Jost94,Jost97,Jost97NPC,Jost98} and harmonic maps from graphs
into Riemannian manifolds or metric spaces have been studied by many
authors, e.g.
\cite{KotaniSunada01,IzekiNayatani05,JostTodjihounde07} and and for alternative definitions, see
\cite{GromovSchoen92,KorevarrSchoen93,KorevarrSchoen97,Sturm01,Sturm05}.

We use our results {concerning the function
theory on graphs to derive various Liouville type theorems for
harmonic maps from graphs}. A particular focus lies on bounded
harmonic maps and harmonic maps of finite energy.

Let $(X,\mu,m)$ be a weighted graph. We briefly recall the set up of Hadamard spaces and harmonic maps.

A complete geodesic space $(Y,d)$ is called an \emph{NPC space} if
it locally satisfies Toponogov's triangle comparison for
non-positive sectional curvature. We refer to Burago-Burago-Ivanov
\cite{Burago01}, Jost \cite{Jost97NPC} and
Bridson-Haefliger\cite{BridsonHaefliger99} for definitions. Here NPC
stands for ``non-positive curvature'' in the sense of Alexandrov.
The space $(Y,d)$ is called an \emph{Hadamard space}, if the
Toponogov's triangle comparison holds globally, i.e., holds for
arbitrary large geodesic triangles. A simply connected NPC space is
an Hadamard space, see \cite{Burago01}. For the sake of simplicity,
we only consider Hadamard spaces, also called $\mathrm{CAT}(0)$
spaces, as targets of harmonic maps $X\to Y$. For general NPC
spaces, we may pass to the universal covers of $X$ and $Y,$ and
consider the equivariant harmonic maps, see Jost
\cite{Jost94,Jost97}.

Let $b: P^1(Y)\to
Y$ denote the \emph{barycenter map} on $Y$, where $P^1(Y)$ is the space of
probability measures on $Y$ with finite first moment, that is,
$b(\nu)$ is the \emph{barycenter} of the probability measure $\nu$ on $Y,$
see e.g.  Sturm \cite[Propositon 4.3]{Sturm03} and confer \cite[Definition~2.2 and Example 1]{KuwaeSturm08}.

We define the random walk measure $P_{x}$ of $x\in X$ by
\begin{align*}
P_x(y):=\frac{\mu_{xy}}{\sum_{z\in X}\mu_{xz}}
\end{align*}
and denote by $u_*P_x$ the push forward of the probability measure
$P_x$ under the map $u:X\to Y$. In order to carry out the barycenter
construction for a map $u$ pointwise, we need $u_*P_x\in P^1(Y)$
which means that $u_*P_x$ has a finite first moment. Thus, similar
to the harmonic function case, we define a class of maps
$$F(X,Y):=\{u:X\to Y\mid \sum_{y\in X}
d(u(y),y_0)P_{x}(y)<\infty\mbox{ for all }x\in X,y_0\in Y \}.$$

\begin{definition}[Harmonic map]\label{d:harmonicmaps}
A map $u:X\to Y$ is called a \emph{harmonic map} if $u\in F(X,Y)$
and for every $x\in X$
$$u(x)=b(u_*P_x).$$
\end{definition}
One immediately finds that the measure $m$ plays no role in the
definition of harmonic maps. {In the following, we always
denote by $u:X\to Y$ a harmonic map from a weighted graph into an
Hadamard space. }

\subsection{Proof of Theorem~\ref{t:harmonicmaps}}\label{s:harmonicmapsproof}

The proof of Theorem~\ref{t:harmonicmaps} is a rather immediate consequence of Theorem~\ref{t:applicationKarp} and the following lemma which is a consequence of Jensen's inequality and convexity of distance functions on Hadamard spaces.

\begin{lemma}\label{l:subharmonic}
For every harmonic map $u$ the functions $X\to[0,\infty)$, $x\mapsto
d(u(x),y)$, for fixed $y\in Y$, are subharmonic.
\end{lemma}
\begin{proof}Jensen's inequality in Hadamard spaces, see
\cite[Theorem~6.2]{Sturm03}, states that for every lower semi-continuous convex
function $g:Y\to [0,\infty)$  and $\nu\in P^1(Y)$
$$g(b(\nu))\leq \int_{Y} g(y)\nu(dy).$$
Now any distance function $y\mapsto d(y,y_0)$ to a point $y_{0}\in
Y$ is  convex in an Hadamard space, see e.g. \cite[Corollary~9.2.14]{Burago01}, which
yields the statement.
\end{proof}

\begin{proof}[Proof of Theorem~\ref{t:harmonicmaps}] Combining Theorem~\ref{t:applicationKarp} and Lemma~\ref{l:subharmonic} yields
that $x\mapsto d(u(x),y)$ is constant. Hence, $u$ is bounded. If
$m\rho_{1}^{-2}(X)=\sum_{x\in X}m(x)\rho_{1}^{-2}(x)=\infty$, then a
constant function is in $L^{p}(X, m\rho_{1}^{-2})$ if and only if it
is zero. Moreover, if $y$ is in the image of $u$ then
$d(u(\cdot),y)\equiv 0$ and hence $u(x)=y$ for all $x\in X$.
\end{proof}

\subsection{Harmonic maps of finite energy}\label{s:harmonicmapsoffiniteenergy}

In this section we consider harmonic maps of finite energy and prove
Theorem~\ref{t:boundedfiniteengery} and
Theorem~\ref{t:boundedfiniteengery2} which are analogues to theorems
of Cheng-Tam-Wang \cite{ChengTamWan96} from Riemannian geometry. We
say a harmonic map $u:X\to Y$  has \emph{finite energy} if
\begin{align*}
\frac{1}{2}\sum_{x,y\in X}\mu_{xy} d^2(u(x),u(y))<\infty.
\end{align*}

In order to do so, we need the equivalence of boundedness of  finite
energy {harmonic functions on a
graph and that of non-negative subharmonic functions}. Recall that a
function $f:X\to\R$ is said to have finite energy if $E(f)<\infty$,
see Section~\ref{s:Laplacians}. In Riemannian geometry such a
theorem was first proven in \cite[Theorem~1.2]{ChengTamWan96}. We
give a different proof here in the discrete setting by Royden's
decomposition.

\begin{thm}\label{t:finiteenergyequivalence} For connected weighted graphs every harmonic function with finite energy is bounded if and only if every non-negative subharmonic function with finite energy is bounded.
\end{thm}
\begin{proof}As positive and negative part of a harmonic function are non-negative and subharmonic functions, boundedness of non-negative subharmonic functions of finite energy implies boundedness of harmonic functions of finite energy. \\
We now turn to the other direction. By
Proposition~\ref{p:recurrence} there are no non-constant subharmonic
functions of finite energy in the case the graph is recurrent.
Therefore, we assume the graph is not recurrent {(also called
transient in the connected case)}. Let $f$ be a non-negative subharmonic function with
finite energy. Then by the discrete version of Royden's
decomposition theorem, see \cite[Theorem~3.69]{Soardi08}, there are
unique functions $g$ and $h$ where $g$ is in the completion of
$C_{c}(X)$ under the norm $\| \ph\|_{o}= (E(\ph)+\ph(o)^{2})^{1/2},$
$\ph\in C_{c}(X)$, and
 $h$ is a harmonic function of finite energy
such that
\begin{align*}
    f=g+h\quad\mbox{and}\quad E(f)=E(g)+E(h)
\end{align*}
By \cite[Lemma~3.70]{Soardi08}, $g\leq 0$ {since $g$ is
subharmonic}. Therefore, $0\leq f\leq h.$ By assumption $h$ is
bounded and, therefore, $f$ is bounded.
\end{proof}

\begin{proof}[Proof of Theorem~\ref{t:boundedfiniteengery}] Let $u:X\to Y$ be a harmonic map of finite energy. For some fixed $y_{0}\in Y$ the function $f=d(u(\cdot),y_{0})$ is non-negative and subharmonic by Lemma~\ref{l:subharmonic}. Furthermore, by the triangle inequality and the assumption that $u$ has finite energy we get
\begin{align*}
    E(f)=\frac{1}{2} \sum_{x,y}\mu_{xy}(f(x)-f(y))^{2}\leq\frac{1}{2}\sum_{x,y\in X}\mu_{xy}d^{2}(u(x),u(y))<\infty.
\end{align*}
Now, by Theorem~\ref{t:finiteenergyequivalence} we get that $f$ as a
non-negative subharmonic function of finite energy must be bounded
whenever every harmonic function of finite energy on $X$ is bounded
(which is our assumption). Thus, $u$ is bounded.
\end{proof}

Theorem~\ref{t:boundedfiniteengery2} is a consequence of
Theorem~\ref{t:boundedfiniteengery} and the  theorem of Kuwae-Sturm
\cite{KuwaeSturm08} below which goes back to Kendall in the manifold
case \cite[Theorem~6]{Kendall88} {(confer
\cite{HuangKendall91,LiWang98,KuwaeSturm08}). However, although it is not explicitly mentioned in \cite{KuwaeSturm08} one actually needs an
additional assumption on the local compactness of the target, i.e.,
every point has a precompact neighborhood.}

\begin{thm}\label{t:Kendallthm}\emph{(Kendall's theorem \cite[Theorem~3.1]{KuwaeSturm08})}
Assume that on a connected weighted graph every  bounded harmonic
functions is constant. Then, every bounded harmonic map into a
locally compact Hadamard space is constant.
\end{thm}

Next, we come to the proof of Theorem~\ref{t:boundedfiniteengery2}.

\begin{proof}[Proof of Theorem~\ref{t:boundedfiniteengery2}]
Let $f$ be a harmonic function on $X$ of finite energy. By a
discrete version of Virtanen's theorem, see
\cite[Theorem~3.73]{Soardi08}, $f$ can be approximated by bounded
harmonic functions $f_{n}$ of finite energy (with respect to the
norm $\|\ph\|_{o}=(E(\ph)+\ph(o)^{2})^{1/2}$). By assumption the functions
$f_{n}$, $n\ge1$, are constant and, thus, $f$ must be constant.
Theorem~\ref{t:boundedfiniteengery} implies now that any harmonic
map is bounded and, thus, Theorem~\ref{t:Kendallthm} implies that
every harmonic map is constant.
\end{proof}


\subsection{Harmonic maps and assumptions on the measure of $X$}\label{s:HarmonicMapsmeasure}
In this subsection we collect several quantitative results that
follow from what we have proven before.

The first corollary can be seen as an analogue to  Yau's $L^{p}$-Liouville type theorem.

\begin{cor} \label{t:infintemeasure_harmonicmaps} Assume a connected weighted graph $X$ allows for a compatible intrinsic metric and let $u$ be a harmonic map into an Hadamard space $Y$. If there is $y\in Y$ such that $d(u(\cdot),y)\in L^{p}(X,m)$ for some $p\in(1,\infty)$, then $u$ is bounded. If additionally $m(X)=\infty$, then $u$ is constant.
\end{cor}
\begin{proof}
The function $d(u(\cdot),y)$ is subharmonic, by
Lemma~\ref{l:subharmonic}, and in $L^{p}(X,m)$ by assumption. Hence,
Corollary~\ref{c:Yau} yields $d(u(\cdot),y)$ is constant. The
assumption of infinite measure implies $d(u(x),y)=0$ for all $x\in
X.$
\end{proof}

We say a harmonic map $u$ into an Hadamard space $(Y,d)$
\emph{grows less than} a function $g:[0,\infty)\to(0,\infty)$ if
$d(u(\cdot),y)$ grows less than $g$ for some $y\in Y$ (confer
Section~\ref{s:applicationKarp}). The next two corollaries are
analogues of Corollary~\ref{c:finitemomentmeasure} and
Corollary~\ref{c:finitemeasure}.

\begin{cor}[Measures with finite moment -- harmonic maps] \label{t:finitemomentmeasure_harmonicmaps} Assume a connected weighted graph allows for a compatible intrinsic metric and the measure has a finite $q$-th moment, $q>-2$. Then every harmonic map into an Hadamard space that grows less that $r\mapsto r^{q+2}$ is constant. In particular, bounded harmonic maps and harmonic maps with finite energy are constant.
\end{cor}
\begin{proof}Let $u$ be a harmonic map.
Since we assume $q+2>0$, we get by the triangle inequality that for all $y\in Y$ the subharmonic (Lemma~\ref{l:subharmonic})
functions $d(u(\cdot),y)$ grow less than $r\mapsto r^{q+2}$. Hence,
by Corollary~\ref{c:finitemomentmeasure} the subharmonic function
$d(u(\cdot),y)$ is constant for all $y$ which implies that $u$ is
constant. This proves the first assertion. Since $q+2>0,$ it is easy
to see that every bounded harmonic function on $X$ is constant. The
second assertion follows from Theorem~\ref{t:Kendallthm} and
Theorem~\ref{t:boundedfiniteengery2}.
\end{proof}

\begin{cor}[Finite measure -- harmonic maps] \label{t:finitemeasure_harmonicmaps} Assume a connected weighted graph allows for a compatible intrinsic metric and $m(X)<\infty.$  Then every harmonic map into an Hadamard space  that grows less than quadratic is constant. In particular, bounded harmonic maps and harmonic maps with finite energy are constant.
\end{cor}
\begin{proof}The statements follow directly by the corollary above putting $q=0$.
\end{proof}

Finally we say that a harmonic map \emph{grows polynomially} if it
grows less than a polynomial and state a corollary analogous to
Corollary~\ref{c:decayingmeasure}.

\begin{thm}[Exponentially decaying measure -- harmonic maps]\label{t:decayingmeasure_harmonicmaps} Assume a connected weighted graph allows for a compatible intrinsic metric, $m(X)<\infty$ and there is $\beta>0$ such that
\begin{align*}
    \limsup_{r\to\infty}\frac{1}{r^{\beta}}\log m(B_{r+1}\setminus B_r) <0.
\end{align*}
Then  every  harmonic map into an Hadamard space that grows polynomially is constant. In particular, bounded harmonic maps and harmonic maps with finite energy are constant.
\end{thm}
\begin{proof}The statements follow from Corollary~\ref{c:decayingmeasure}, Theorem~\ref{t:Kendallthm} and Theorem~\ref{t:boundedfiniteengery2}.
\end{proof}

\textbf{Acknowledgement}. {BH thanks J\"urgen Jost for
inspiring discussions on $L^p$ Liouville theorem and constant
support, and acknowledges the financial support from the funding of
the European Research Council under the European Union's Seventh
Framework Programme (FP7/2007-2013) / ERC grant agreement
n$^\circ$~267087.} MK enjoyed discussions with Gabor Lippner, Dan
Mangoubi, Marcel Schmidt and Rados{\l}aw Wojciechowski on the
subject and acknowledges the financial support of the German Science
Foundation (DFG), Golda Meir Fellowship, the Israel Science
Foundation (grant no. 1105/10 and  no. 225/10) and BSF grant no.
2010214.

\bibliography{Lqharmonic2}
\bibliographystyle{alpha}

\end{document}